\documentclass[12pt]{article}
\usepackage{amsmath, amsfonts, amsthm, amssymb}
\usepackage{times}
\usepackage{a4wide}

\def \Re {{\rm Re}}
\def \Im {{\rm Im}}
\def \Tr {{\rm Tr}}
\def \RR {\mathbb R}

\def \EE {\mathbb E}

\def \CC {\mathbb C}
\def \PP {\mathbb P}

\def \eps {\varepsilon}
\def \vphi {\varphi}

\def \cQ {\mathcal Q}

\def \cF {\mathcal F}

\def \cG {\mathcal G}

\newtheorem{theorem}{Theorem}[section]
\newtheorem{lemma}[theorem]{Lemma}
\newtheorem{proposition}[theorem]{Proposition}
\newtheorem{corollary}[theorem]{Corollary}
\newtheorem{remark}[theorem]{Remark}

\def\myffrac#1#2 in #3{\raise 2.6pt\hbox{$#3 #1$}\mkern-1.5mu\raise 0.8pt\hbox{$
#3/$}\mkern-1.1mu\lower 1.5pt\hbox{$#3 #2$}}

\def\qed{\hfill $\vcenter{\hrule height .3mm
\hbox {\vrule width .3mm height 2.1mm \kern 2mm \vrule width .3mm
height 2.1mm} \hrule height .3mm}$ \bigskip}

\begin{document}

\title{Eldan's stochastic localization and tubular neighborhoods of complex-analytic sets}
\author{Bo'az Klartag}
\date{}
\maketitle

\abstract{Let $f: \CC^n \rightarrow \CC^k$ be a holomorphic function and set $Z = f^{-1}(0)$. Assume that $Z$ is non-empty. We prove that for any $r > 0$,
$$ \gamma_n(Z + r) \geq \gamma_n(E + r), $$
where $Z + r$ is the Euclidean $r$-neighborhood of $Z$, where $\gamma_n$ is the standard Gaussian measure in $\CC^n$,
and where $E \subseteq \CC^n$ is an $(n-k)$-dimensional, affine, complex subspace
whose distance from the origin is the same as the distance of $Z$ from the origin. }

\section{Introduction}

Eldan's stochastic localization is a rather recent analytic technique that has emerged from convex geometry.
 In a nutshell, the idea is to construct a time-parameterized
family of decompositions of
a given probability measure on $\RR^n$ into a mixture of increasingly curved measures,
in the sense that their
densities become more log-concave with time.

\smallbreak The technique was used by Eldan \cite{Eldan1} and by Lee and Vempala \cite{LV} in their works on the
KLS conjecture, by Eldan in his works on noise sensitivity \cite{Eldan2} and on Skorokhod embeddings
\cite{Eldan3}, by Eldan and Lehec \cite{Eldan_lehec} in connection with Ball's thin-shell problem and Bourgain's slicing conjecture,
and by Ding, Eldan and Zhai \cite[Theorem 1.6]{eldan_more} for studying the concentration of a supremum of a Gaussian process.

\smallbreak In this paper we apply this technique to the study of complex-analytic subvarieties of $\CC^n$.
Write $\gamma_n$ for the standard Gaussian probability measure in $\CC^n$, whose density is $z \mapsto (2\pi)^{-n} \exp(-|z|^2/2)$,
where $|z| =\sqrt{\sum_j |z^j|^2}$ for $z = (z^1,\ldots,z^n) \in \CC^n$. For a non-empty set $S \subseteq \CC^n$ we write
$$ d(0, S) = \inf_{z \in S} |z| $$
for the Euclidean distance from the origin to the set $S$.

\begin{theorem}
Let $1 \leq k \leq n$, let $f: \CC^n \rightarrow \CC^k$ be a holomorphic map and set $Z = f^{-1}(0)$.
Assume that $Z$ is non-empty. Let $E \subseteq \CC^n$
be an $(n-k)$-dimensional, affine, complex subspace with $d(0, E) = d(0, Z)$. Then,
\begin{equation}  \gamma_n( Z + r ) \geq  \gamma_n( E + r ) \qquad \qquad \text{for all} \ r > 0, \label{eq_511} \end{equation}
where $Z + r = \{ z + w \, ; \, z \in Z, w \in \CC^n, |w| \leq r \}$.
 \label{thm_1048}
\end{theorem}

Intuitively, a consequence of Theorem \ref{thm_1048} is that a fiber of a holomorphic map
cannot escape from the origin faster than a flat that has the same distance from the origin. The limit case $r \rightarrow 0$
in (\ref{eq_511}) implies that
$$ \int_Z \frac{e^{-|z|^2/2}}{(2 \pi)^n} \geq \int_E \frac{e^{-|z|^2/2}}{(2 \pi)^n} = \frac{e^{-d(0, Z)^2/2}}{(2 \pi)^k}, $$
where both integrals are carried out with respect to the Hausdorff measure of real dimension $2(n-k)$. By specializing
Theorem \ref{thm_1048} to the case where $0 \in Z$ we obtain the following:

\begin{corollary} Let $1 \leq k \leq n$ and let $f: \CC^n \rightarrow \CC^k$ be a holomorphic map with $f(0) = 0$. Set $Z =f^{-1}(0)$. Then,
\begin{equation}
\gamma_n( Z + r ) \geq  \gamma_n( \CC^{n-k} + r ) \qquad \qquad \text{for all} \ r > 0,  \label{eq_1608}
\end{equation}
where we think of $\CC^{n-k}$ as the subspace of $\CC^n$ of all vectors whose last $k$ coordinates vanish.
 \label{thm_1237_}
\end{corollary}

We view Corollary \ref{thm_1237_} in the context of
Gromov's Gaussian waist inequality \cite{gromov}. If the function $f$ in Corollary \ref{thm_1237_} is assumed to be
merely continuous, then Gromov showed that there exists some $t \in \CC^k$ such that the set $Z = f^{-1}(t)$ satisfies (\ref{eq_1608}).
From Corollary \ref{thm_1237_} we thus gather  that when $f: \CC^n \rightarrow \CC^k$ is not only continuous but also holomorphic, it is always possible to select $t = f(0)$
in the Gaussian waist inequality.

\smallbreak Of course, the complex structure is irrelevant for Gromov's Gaussian waist inequality, which is usually formulated
for continuous maps $f: \RR^n \rightarrow \RR^k$.
With the exception of the case $k=1$,
very little is known about the special value $t$ for which the fiber $Z = f^{-1}(t)$ is large. A random choice
of $t$ is worthless in general, as explained by Alpert and Guth \cite{HG}. The idea that $t = f(0)$ should
be suitable for a holomorphic $f: \CC^n \rightarrow \CC^k$ is rather natural, in view of the following statement (see
Lelong \cite{Lelong}, Rutishauser \cite{Ruti} or e.g. Chirka \cite[Section 15.1]{Chirka}): Under the assumptions of Corollary \ref{thm_1237_}, the function
$$ r \mapsto \frac{Vol_{2(n-k)} (Z \cap r B^n)}{Vol_{2(n-k)}( r B^{2(n-k)} )} $$
is a non-decreasing function of $r > 0$, whose limit as $r \rightarrow 0^+$ is at least one.
Here $B^n = \{ z \in \CC^n \, ; \, |z| \leq 1 \}$ and $Vol_{2(n-k)}$ is the $2(n-k)$-dimensional
Hausdorff measure. We remark that this statement
may be used to show that $$ \liminf_{r \rightarrow 0^+} \left[ \gamma_n( Z + r ) / \gamma_n( \CC^{n-k} + r ) \right] \geq 1. $$
We thus obtain a rather direct proof of a limit case of Corollary \ref{thm_1237_}.
A corresponding limit case of Theorem \ref{thm_1048} may be deduced from the results of Brendle and Hung \cite{BH}
and Alexander, Hoffman and Osserman \cite{AHO},
even in the more general setting of a minimal submanifold. Still, the only proof of Theorem \ref{thm_1048} or Corollary \ref{thm_1237_}
of which we are aware is quite indirect and it involves Eldan's stochastic localization. It would be interesting to find a proof of Theorem \ref{thm_1048} which does not involve stochastic processes.

\smallbreak
We proceed with a non-Euclidean  waist inequality for scalar holomorphic functions $f: \CC^n \rightarrow \CC$ that follows from Theorem \ref{thm_1048}.
By a convex body we mean  a compact, convex set with a non-empty interior.
We say that a convex body $K \subseteq \CC^n$ is circled if $e^{i t} K = K$ for all $t \in \RR$. Here, $\lambda K = \{ \lambda z \, ; \, z \in K\}$, $A + B = \{ a + b \, ; \, a \in A, b \in B \}$ and $x + A = \{ x \} + A$.
Equivalently, a subset $K \subseteq \CC^n$ is a circled convex body if and only if $K$ is the unit ball of a complex norm on $\CC^n$.
Thus the following corollary is concerned with tubular neighborhoods with respect to an arbitrary complex norm on $\CC^n$.

\begin{corollary} Let $K \subseteq \CC^n$ be a circled convex body.
Then there exists a complex, $(n-1)$-dimensional, linear subspace $H \subseteq \CC^n$ with the following property:
Let  $f: \CC^n \rightarrow \CC$ be holomorphic, set $Z = f^{-1}(0)$ and assume that $Z$ is non-empty.
Then for any $r > 0$,
$$
\gamma_n( Z + r K ) \geq \gamma_n( H_1 + r K),
$$
where $H_1 \subseteq \CC^n$ is a translate of $H$ with $d(0, H_1) = d(0, Z)$.
 \label{thm_1237}
\end{corollary}

Let us now discuss the proof of Theorem \ref{thm_1048} in the case where $0 \in Z$. By approximation, it suffices to consider the case where $0$ is a regular value of $f$,
thus $Z = f^{-1}(0) \subseteq \CC^n$
is an $(n-k)$-dimensional complex submanifold of $\CC^n$. We fix a probability space $(\Omega, \cF, \PP)$,
which for concreteness is set to be the unit interval $\Omega = [0,1]$ equipped with the Lebesgue measure. With $\PP$-almost any $\omega \in \Omega$ we will associate a Gaussian measure $\mu_{\omega}$ on $\CC^n$ such that
\begin{equation}  \gamma_n(A) = \int_{\Omega} \mu_{\omega}(A) d \PP(\omega) = \EE_{\omega} \mu_{\omega}(A) \qquad \text{for any measurable set} \ A \subseteq \CC^n. \label{eq_935}
\end{equation}
The Gaussian probability measure $\mu_{\omega}$ has several important properties: It is almost-surely centered at a certain point of
 $Z$, it is supported on a $k$-dimensional affine
subspace of $\CC^n$, and it is more curved than the standard Gaussian measure in $\CC^n$.
The measure decomposition in (\ref{eq_935}),
which is formulated precisely in Theorem \ref{thm_330} below, is the main technical result in this paper.

\smallbreak The decomposition (\ref{eq_935}) is obtained as the limit $t \rightarrow \infty$ of a stochastic process of decompositions parameterized by time.
Recall that a stochastic process is a family of random variables $(X_t)_{t \geq 0}$, i.e., a family of functions defined
on $\Omega$. All of our stochastic processes are continuous, meaning that for almost any $\omega \in \Omega$, the map $t \mapsto X_t(\omega)$ is continuous in $[0, \infty)$.
From now on in this paper, we will supress the dependence on the sample
point $\omega \in \Omega$, and replace integrals with respect to $\PP$ by the expectation sign $\EE$.

\smallbreak Let us briefly explain the relation of holomorphic functions to stochastic processes and in particular to It\^o martingales.
Recall the well-known fact, that $f(W_t)$ is a local martingale whenever $f$ is harmonic
and $(W_t)_{t \geq 0}$ is a Brownian motion.
When $f: \CC^n \rightarrow \CC$ is holomorphic (pluriharmonic is sufficient), the stochastic process $f(X_t)$
is a local martingale whenever $$ d X_t = \Sigma_t d W_t, $$ where $\Sigma_t \in \CC^{n \times n}$ is an arbitrary
adapted process. This provides much more flexibility, which allows us to control the center of the Gaussian measure $\mu_t$
and confine it to the manifold $Z$. The details are below.

\smallbreak For $z, w \in \CC^n$ we write $z \cdot w = \sum_j z^j w^j$. Note that in our notation
the vector $w$ is not conjugated. We use $| A |$ to denote the Hilbert-Schmidt norm of a matrix $A \in \CC^{n \times n}$,
and $Id$ is the identity matrix. For Hermitian matrices $A, B \in \CC^{n \times n}$, we write $A \leq B$ if the difference $B - A$ is positive semi-definite,
while $A > 0$ means that $A$ is positive-definite. We write $\log$ for the natural logarithm. By a smooth function we mean $C^{\infty}$-smooth.

\smallbreak
\emph{Acknowledgements.} I would like to thank Bo Berndtsson, Ronen Eldan, Sasha Logunov and Sasha Sodin for interesting related discussions.
Thanks also to the anonymous referee for comments simplifying the proof of  Lemma \ref{lem_955}.
Supported by a grant from the European Research Council (ERC).

\section{Eldan's stochastic localization}

Denote by $M_n^+(\CC)$ the collection of all Hermitian $n \times n$ matrices that are positive-definite.
Write $\cQ^n$ for the space of all quadratic polynomials $P: \CC^n \rightarrow \RR$ of the form
\begin{equation} p(z) = \frac{(z - a)^* B (z - a)}{2} - \log \det(B) \qquad \qquad \text{with} \ B \in M_n^+(\CC), a \in \CC^n.
\label{eq_929} \end{equation}
Here $z = (z^1,\ldots, z^n) \in \CC^n$ is actually viewed as a column vector, and $z^*$ is the row vector with entries
$\overline{z^1},\ldots, \overline{z^n}$. The representation (\ref{eq_929}) of  $p \in \cQ^n$ clearly
determines the vector $a_p = a$ and the positive-definite Hermitian matrix $B_p = B(p) = B$. Note that for any $p \in \cQ^n$,
\begin{equation}  \int_{\CC^n} e^{-p(z)} d \lambda(z) = \det(B) \int_{\CC^n} e^{-(z - a)^* B (z - a) / 2 } d \lambda(z) = \int_{\CC^n} e^{-|w|^2 / 2} d \lambda(w) = (2 \pi)^n,
\label{eq_941} \end{equation}
where $\lambda$ is the Lebesgue measure in $\CC^n$.

\smallbreak
Let $(W_t)_{t \geq 0}$ be a standard Brownian motion in $\CC^n$ with $W_0 = 0$ which is defined on our probability space $(\Omega, \cF, \PP)$. This means that $\Re(W_t)$ and $\Im(W_t)$ are two independent standard Brownian motions in $\RR^n$, so in particular $\EE |W_t|^2 = 2 nt$ for all $t \geq 0$. Recall that the quadratic variation of two continuous functions $\vphi, \psi: [0, t] \rightarrow \CC$ is
$$ [\vphi, \psi]_t = \lim_{\eps(P) \rightarrow 0} \sum_{i=1}^{N_P} (\vphi(t_i) - \vphi(t_{i-1})) \cdot (\psi(t_i) - \psi(t_{i-1}))  $$
whenever the limit exists, where $P = \{0 = t_0 < t_1 < \ldots < t_{N_P} = t \}$ is a partition of  $[0,t]$ into $N_P$ intervals and $\eps(P) = \max_{1 \leq i \leq N_P} |t_i - t_{i-1}|$
is the mesh of the partition.
Note that the quadratic variation of the Brownian motion is a deterministic function, namely
\begin{equation}  [W^j, W^k]_t = [\overline{W^j}, \overline{W^k}]_t = 0, \quad [W^j, \overline{W^k}]_t = 2 t \delta_{jk} \qquad (t > 0) \label{eq_306}
\end{equation}
where $W_t = (W_1^t, \ldots, W_n^t) \in \CC^n$ and $\delta_{jk}$ is Kronecker's delta.
The following proposition claims the existence a certain stochastic process, adapted to the filtration induced by the Brownian motion and attaining  values in
the space $\cQ^n$.

\begin{proposition}[``existence of the Eldan process''] Assume that $\Sigma: \cQ^n \rightarrow \CC^{n \times n}$ is a smooth matrix-valued map with $ |B_p^{1/2} \Sigma(p)| \leq 1$ for all $p \in \cQ^n$.
Then there exists a  $\cQ^n$-valued adapted stochastic process $(p_t)_{t \geq 0}$
such that the following hold:
\begin{enumerate}
\item[(i)] $\displaystyle p_0(z) = |z|^2 / 2$ for all $z \in \CC^n$.
\item[(ii)] For any measurable set $S \subseteq \CC^n$ and for any $t > 0$,
$$ \int_{S} e^{-p_0(z)} d \lambda(z) = \EE \int_{S} e^{-p_t(z)} d \lambda(z).  $$
\item[(iii)] Denote $a_t = a_{p_t}, B_t = B_{p_t}$ and $\Sigma_t = \Sigma(p_t)$. Then the following system of stochastic differential equations hold true in $t \in (0, +\infty)$:
\begin{equation} d a_t = \Sigma_t d W_t \qquad \text{and} \qquad B_t = B_t \Sigma_t \Sigma_t^* B_t dt.
\label{eq_826} \end{equation}
\end{enumerate} \label{prop_1345}
\end{proposition}

\begin{proof} Equation (\ref{eq_826}) is a stochastic differential equation with smooth coefficients in $a$ and $B$.
According to the standard theory (e.g., Kallenberg \cite[Section 21]{Kal})
this equation has a unique strong solution for all $0 \leq t < T$ with the initial condition $B_0 = Id, a_0 = 0$, where the random variable $T \in [0, +\infty]$ is a stopping time which is almost-surely non-zero.
The stopping time $T$ is the explosion time of the process, in the sense that $T = \sup_{k \geq 1} T_k$ where
\begin{equation}  T_k = \inf \{ t \geq 0 \, ; \, |a_t| + | B_t | \geq k \} \qquad \qquad (k=1,2,\ldots). \label{eq_1702} \end{equation}
Let us  show that  $\PP(T = +\infty) = 1$.
We first explain why $B_t$ cannot blow up in finite time.
In fact, we will show that $B_t$ is bounded by a deterministic function of $t$.
For $k \geq 1$ and $t > 0$ we abbreviate $t_k = \min \{t, T_k \}$. We know that  $B_0 = Id$ and by (\ref{eq_826}),
\begin{equation}  B_{t_k} = Id + \int_0^{t_k}  B_s \Sigma_s \Sigma_s^* B_s ds \qquad \qquad (t > 0). \label{eq_1656}
\end{equation}
In particular, $B_{t_k} \geq Id$ for all $t > 0$. Since  $|B_t^{1/2} \Sigma_t| \leq 1$, we deduce from (\ref{eq_1656}) that
$$ \Tr [B_{t_k}] = n + \int_0^{t_k}  \Tr [ B_s^{1/2} (B_s^{1/2} \Sigma_s \Sigma_s^* B_s^{1/2}) B_s^{1/2} ] ds \leq n +  \int_0^{t_k} \Tr[B_s] ds, $$
where $\Tr[A]$ is the trace of the matrix $A \in \CC^{n \times n}$.
Thus the Gr\"onwall inequality (e.g., \cite[Lemma 21.4]{Kal})
implies that $\Tr[B_{t_k}] \leq n e^{t}$ for all $t > 0$ and $k \geq 1$. Since $B_{t_k} \geq Id$, the bound on the trace
shows that $B_t$ cannot blow up in finite time.
Next, the process $(a_t)_{t \geq 0}$ is not necessarily bounded by a deterministic function, yet for any $s \in (0,t_k)$
we have that $|\Sigma_s|^2 = \Tr[\Sigma_s \Sigma_s^* ] \leq \Tr[B_s \Sigma_s \Sigma_s^*] = |B_s^{1/2} \Sigma_s|^2 \leq 1$. Hence, by the It\^o isometry,
$$ \EE |a_{t_k}|^2 = 2 \EE \int_0^{t_k} |\Sigma_s|^2 ds \leq 2 t_k \leq 2 t.  $$
Now Doob's martingale inequality (e.g., \cite[Proposition 7.16]{Kal}) shows that $\PP( \sup_{0 \leq s \leq t_k} |a_{s}| > k/2) \longrightarrow 0$ as $k \rightarrow \infty$, for any fixed $t > 0$. Consequently,
$$ \lim_{k \rightarrow \infty} \PP(T_k < t) = 0 \qquad \text{for all} \ t > 0. $$
It follows that $ T = +\infty$ with probability one. We have thus shown that the $\cQ^n$-valued stochastic process $(p_t)_{t \geq 0}$ is well-defined
for all $t > 0$, and moreover (i) and (iii) hold true. It remains to prove (ii). We begin by deducing from (\ref{eq_826}) that
\begin{equation}  \frac{d \log \det(B_t)}{dt} = \Tr \left[ B_t^{-1} \frac{d B_t}{dt} \right] = \Tr[ \Sigma_t \Sigma_t^* B_t ] = \Tr[ \Sigma_t^*  B_t  \Sigma_t]. \label{eq_1420}
\end{equation}
Fix $z \in \CC^n$. Note that as $|B_t| \leq n e^t$ and $|\Sigma_t| \leq 1$, for any fixed $t > 0$,
\begin{equation}
 \EE \int_0^t (z - a_s)^* B_s \Sigma_s \Sigma_s^* B_s (z - a_s) ds \leq n^2 e^{2t} \cdot \EE  \int_0^t |z - a_s|^2 ds
 < \infty.
 \label{eq_1418_}
 \end{equation}
Hence the process $L_t = L_t(z) =  \Re \left \{  \int_0^t (z - a_s)^* B_s \Sigma_s d W_s \right \}$ is a martingale
whose quadratic variation process satisfies $d [L, L]_t =  (z - a_t)^* B_t \Sigma_t \Sigma_t^* B_t (z - a_t) dt$.
We claim that
\begin{equation}\label{eq_946}
d p_t(z) =  - d L_t + d [L, L]_t / 2 \qquad \qquad (t > 0).
\end{equation}
Indeed, for $j=1,\ldots,n$  write $\Sigma_t^j$ for the $j^{th}$ row of the matrix $\Sigma_t$. Then
by (\ref{eq_306}) and (\ref{eq_826}) the quadratic variation process $[a^j, \overline{a^k}]_t$ satisfies
\begin{equation}  d [a^j, \overline{a^k}]_t = 2 \Sigma_t^j (\Sigma_t^k)^*   dt. \label{eq_307} \end{equation}
Note that $[a^j,a^k]_t = [\overline{a^j}, \overline{a^k}]_t = 0$, and also that $B_t$ is a process of finite variation.
By It\^o's formula, for any fixed $z \in \CC^n$,
\begin{align}  \nonumber
 2 d p_t(z) & = ( z - a_t)^* d B_t (z - a_t) -  (d a_t)^* B_t (z - a_t) -
 (z -a_t)^* B_t d a_t \\ & - 2 d \log \det(B_t) + 2 \Tr[\Sigma^*_t B_t \Sigma_t ] dt. \label{eq_1447}
\end{align}
Now (\ref{eq_946}) follows from (\ref{eq_826}), (\ref{eq_1420}) and (\ref{eq_1447}). From (\ref{eq_946}) and the It\^o formula,
$$ d e^{-p_t(z)} = e^{-p_t(z)} \cdot d L_t. $$
Hence $(e^{-p_t(z)})_{t \geq 0}$ is a local martingale. Since $e^{-p_t(z)} \leq \det(B_t) \leq n^n e^{t n}$,
the non-negative process $(e^{-p_t(z)})_{t \geq 0}$ is in fact a martingale.
 By the martingale property, $\EE e^{-p_t(z)} = e^{-p_0(z)}$ for all $t$ and $z$.
In order to prove (ii), observe that for any $t > 0$,
$$ \EE \int_{S} e^{-p_t(z)} d \lambda(z) = \int_S \EE e^{-p_t(z)} d \lambda(z) = \int_S e^{-p_0(z)} d \lambda(z) < \infty. $$
The application of Fubini's theorem in the space $\CC^n \times \Omega$, where $\Omega$ is the probability space on which the Brownian motion $(W_t)_{t \geq 0}$ is defined,
is legitimate as the integrand is non-negative.
This completes the proof of the proposition.
\end{proof}

Our next lemma shows that with probability one,
the center $a_t$ of the probability density $(2 \pi)^{-n} e^{-p_t(z)}$ does not escape to infinity as $t \rightarrow \infty$,
and in fact almost-surely it  has a finite limit in $\CC^n$.

\begin{lemma} We work under the notation and assumptions of Proposition \ref{prop_1345}. Define
 $$ a_{\infty} := \lim_{t \rightarrow \infty} a_t \qquad \text{and} \qquad
A_{\infty} := \lim_{t \rightarrow \infty} 2 B_t^{-1}. $$
Then these limits are well-defined almost-everywhere in our probability space $(\Omega, \cF, \PP)$.
Moreover, the random point $a_{\infty} \in \CC^n$ satisfies $\EE |a_{\infty}|^2 \leq 2n$
while the random Hermitian matrix $A_{\infty}$ satisfies $0 \leq A_{\infty} \leq 2 \cdot Id$ almost-surely.  \label{lem_220}
\end{lemma}

\begin{proof} According to (\ref{eq_826}),
\begin{equation}  \frac{d B_t^{-1}}{dt} = -B_t^{-1} \frac{d B_t}{dt} B_t^{-1} = -B_t^{-1} (B_t \Sigma_t \Sigma_t^* B_t) B_t^{-1} = - \Sigma_t \Sigma_t^*. \label{eq_219}
\end{equation}
Recall also that $B_0 = Id$ with $B_t \geq Id$ for all $t$. Therefore, for any $t > 0$,
\begin{equation} \int_0^{t} \Sigma_s \Sigma_s^* ds = - \int_0^{t} \frac{d B_s^{-1}}{ds} ds = Id - B_t^{-1} \leq Id. \label{eq_1137}
\end{equation}
Since $d a_t = \Sigma_t d W_t$, we deduce from (\ref{eq_1137}) that  the quadratic variation processes of the martingales $\Re[a_t^j]$ and $\Im[a_t^j]$
are bounded by $2n$ for all $t$ and $j$. Hence, according to Doob's margingale convergence theorem (e.g., \cite[Chapter 7]{Kal}),
\begin{equation}  a_{\infty} = \lim_{t \rightarrow \infty} a_t \label{eq_1140} \end{equation}
is a well-defined random vector in $\CC^n$, where the convergence in (\ref{eq_1140}) is both almost-everywhere in $\Omega$ and in the sense of $L^2$. Moreover, by (\ref{eq_1137}),
\begin{equation*} \EE |a_{\infty}|^2 = 2 \EE \int_0^{\infty} |\Sigma_s|^2 ds \leq 2n.
\end{equation*}
Finally, the positive-definite Hermitian matrix $B_t^{-1}$ satisfies $B_0^{-1} = Id$ and $
d B_t^{-1} / dt \leq 0$, according to (\ref{eq_219}). Thus the limit $A_{\infty} = 2 \cdot \lim_{t \rightarrow \infty} B_t^{-1}$ exists almost-surely, and it satisfies $0 \leq A_{\infty} \leq 2 \cdot Id$.
\end{proof}

Recall that $A \in \CC^{n \times n}$ is an orthogonal projection if $A^2 = A = A^*$.
Write $G_{n,k}$ for the Grassmannian of all $k$-dimensional complex subspaces of $\CC^n$.
For a subspace $E \in G_{n,k}$ we denote by $\pi_E$  the orthogonal projection matrix
whose {\it kernel} equals $E$.
Given a Hermitian matrix $A \in \CC^{n \times n}$ we denote its eigenvalues by $\lambda_1(A) \leq \ldots \leq \lambda_n(A)$.
We write $\overline{M_n^+(\CC)}$ for the set of all Hermitian $n \times n$ matrices that are positive semi-definite.
Proposition \ref{prop_1345} states that
\begin{equation}  B_t = Id + \int_0^t B_s \Sigma_s \Sigma_s^* B_s ds \qquad \qquad (t \geq 0). \label{eq_1024}
\end{equation}
Since $t \mapsto B_t \Sigma_t \Sigma_t^* B_t$ is continuous, then with probability one, the map $t \mapsto B_t$  is continuously differentiable in $(0, \infty)$.
Set $M_t = B_t^{1/2} \Sigma_t \Sigma_t^* B_t^{1/2}$. Later on we will select a suitable map $\Sigma: \cQ^n \rightarrow \CC^{n \times n}$
such that with probability one, $B_t$ and $M_t$ will satisfy
the assumptions of the following standard, non-probabilistic lemma.

\begin{lemma}[``eigenvalue growth''] Let $0 \leq k \leq n-1, \eps > 0$ and let $B, M: [0, \infty) \rightarrow \overline{M_n^+(\CC)}$ be two functions, with $M$ continuous
and $B$ continuously differentiable. Assume that for all $t \geq 0$,
\begin{equation} \lambda_{k+1}(M_t) \geq \eps \qquad \text{and} \qquad \frac{d B_t}{dt} = B_t^{1/2} M_t B_t^{1/2} \label{eq_925} \end{equation}
where $M_t = M(t)$ and $B_t = B(t)$. Assume also that $B_0 = Id$.
Then for all $t > 0$,
\begin{equation}  \lambda_{k+1}(B_t) \geq \frac{\eps t}{k+1}. \label{eq_953} \end{equation}
\label{lem_955}
\end{lemma}

\begin{proof} We recall the min-max characterization of the eigenvalues of a Hermitian $n \times n$ matrix $A$:
\begin{equation}  \lambda_{k+1}(A) = \min_{E \in G_{n, k+1}} \max_{0 \neq v \in E} \frac{v^* A v}{|v|^2}. \label{eq_1723}
\end{equation}
Set $\ell = n-(k+1)$. For any $F \in G_{n, \ell}$ and $A \in \overline{M_n^+(\CC)}$,
\begin{equation} \Tr[A \pi_F] \geq \max_{0 \neq v \in F^{\perp}} \frac{v^* A v}{|v|^2} \geq \min_{E \in G_{n, k+1}} \max_{0 \neq v \in E} \frac{v^* A v}{|v|^2} = \lambda_{k+1}(A),
\label{eq_754} \end{equation}
where $F^{\perp} = \{ z \in \CC^n \, ; \, \forall w \in F, \ w^* z = 0 \}$ is the orthogonal complement.
It also follows from (\ref{eq_1723}) that for any $A, C \in M_n^+(\CC)$ with $C \geq Id$,
\begin{align} \lambda_{k+1}(C^* A C) & = \min_{E \in G_{n, k+1}} \max_{0 \neq v \in E} \frac{(C v)^* A (C v)}{|v|^2} \label{eq_756} \\ & \geq \min_{E \in G_{n, k+1}} \max_{0 \neq v \in E} \frac{(C v)^* A (C v)}{|C v|^2}
%\stackrel{``w = Cv"}
{=} \min_{F \in G_{n,k+1}} \max_{0 \neq w \in F} \frac{w^* A w}{|w|^2} = \lambda_{k+1}(A). \nonumber \end{align}
Note that (\ref{eq_925}) implies the inequality $B_t \geq B_0 = Id$ for all $t \geq 0$. Consequently, we deduce from (\ref{eq_754})
and (\ref{eq_756}) that for any $t \geq 0$ and $F \in G_{n,\ell}$,
\begin{equation}  \frac{d}{dt} \Tr[B_t \pi_F] =  \Tr[B_t^{1/2} M_t B_t^{1/2} \pi_F] \geq \lambda_{k+1}( B_t^{1/2} M_t B_t^{1/2} ) \geq \lambda_{k+1}(M_t) \geq \eps.
\label{eq_925_} \end{equation}
We integrate (\ref{eq_925_}) and obtain that for $t > 0$ and $F \in G_{n,\ell}$,
\begin{equation}
\Tr[B_t \pi_F] \geq \Tr[B_0 \pi_F] + \eps t = k+1 + \eps t.
\label{eq_926}
\end{equation}
In particular, (\ref{eq_926}) holds true for the subspace $F = F_t$ that is spanned by the eigenvectors
of $B_t$ that corresponds to the eigenvalues $\lambda_{k+2}(B_t), \ldots, \lambda_n(B_t)$. We conclude from (\ref{eq_926}) that
$$
k+1 + \eps t \leq \Tr[B_t \pi_{F_t}] = \Tr[\pi_{F_t} B_t \pi_{F_t}] = \sum_{j=1}^{k+1} \lambda_{j}(B_t) \leq (k+1) \lambda_{k+1}(B_t),
$$
completing the proof of (\ref{eq_953}).
\end{proof}

It is certainly possible to perform a more accurate analysis and improve the estimate of Lemma \ref{lem_955},
yet it is not needed here. Once we know that $\lambda_{k+1}(B_t)$ is large, the density $(2 \pi)^{-n} e^{-p_t(z)}$ is going to be
concentrated on a small neighborhood of a $k$-dimensional affine subspace in $\CC^n$.

\smallbreak A complex Gaussian probability measure in $\CC^n$, or a complex Gaussian in short, is a probability measure $\mu$
supported in a complex,  affine subspace $E \subseteq \CC^n$ with density in $E$ that is proportional to
\begin{equation}  z \mapsto e^{-(z-a)^* B (z-a)/2} \qquad \qquad (z \in E) \label{eq_229} \end{equation}
for some $a \in E$ and a positive-definite, Hermitian operator $B$. We say that $\mu$ is {\it more curved than the standard
Gaussian} if $$ B \geq Id. $$ Write $\cG^n$ for the collection of all complex Gaussians in $\CC^n$
equipped with the topology of weak convegrence, i.e., $\mu_m \longrightarrow \mu$ if
$$ \forall \vphi \in C(\CC^n), \qquad \int_{\CC^n} \vphi d \mu_m \stackrel{m \rightarrow \infty} \longrightarrow \int_{\CC^n} \vphi d \mu. $$
Here, $C(\CC^n)$ is the collection of all bounded, real-valued, continuous functions on $\CC^n$.
Note that for any quadratic polynomial $p \in \cQ^n$, we deduce from (\ref{eq_941}) that the measure
with density $(2 \pi)^{-n} e^{-p}$ belongs to $\cG^n$.
By a {\it random complex Gaussian} we mean a random variable
attaining values in $\cG^n$. For $\mu \in \cG^n$ and $j,k=1,\ldots,n$ we denote
$$ a_{\mu} := \int_{\CC^n} z d \mu, \qquad A_{\mu}^{j\bar{k}} = \int_{\CC^n} z^j \overline{z^k} d \mu(z) \, - \, \int_{\CC^n} z^j  d \mu(z) \int_{\CC^n} \overline{z^k}  d \mu(z).  $$
Thus $a_{\mu} \in \CC^n$ is the center of $\mu$ while $A_{\mu} = (A_{\mu}^{j\bar{k}})_{j,k=1,\ldots,n} \in \overline{M_n^+(\CC)}$ is the complex covariance matrix.
Clearly $\mu \in \cG^n$ is determined by $a_{\mu}$ and $A_{\mu}$, and in fact
under the representation (\ref{eq_229}) we have $$ A_{\mu} = 2 B^{-1} $$ in the case where $E = \CC^n$.
A standard argument shows that if $\mu_1,\mu_2,\ldots \in \cG^n$ satisfy
$$ a_{\mu_m} \stackrel{m \rightarrow \infty}\longrightarrow a \in \CC^n, \qquad A_{\mu_m} \stackrel{m \rightarrow \infty}{\longrightarrow} A \in
\overline{M_n^+(\CC)}, $$
then $\mu_m \longrightarrow \mu$  where $\mu$ is a complex Gaussian in $\CC^n$ with $a_{\mu} = a$ and $A_{\mu} = A$.

\begin{proposition} We work under the notation and assumptions of Proposition \ref{prop_1345}. Let $\eps > 0$ and assume that
with probability one,
\begin{equation} \forall t \geq 0, \qquad \lambda_{k+1}(M_t) \geq \eps, \label{eq_239} \end{equation}
where $M_t = B_t^{1/2} \Sigma_t \Sigma_t^* B_t^{1/2}$.
Denote by $\mu_t$ the measure on $\CC^n$ whose density is $z \mapsto (2 \pi)^{-n} e^{-p_t(z)}$.

\smallbreak Then $\mu_t \in \cG^n$ for all $t \geq 0$ and
$ \mu_{\infty} := \lim_{t \rightarrow \infty} \mu_t $
is a well-defined random complex Gaussian. Moreover, $\mu_{\infty}$ is supported in a $k$-dimensional, complex, affine subspace, it is more curved
than the standard Gaussian measure and
\begin{equation}  \forall \vphi \in L^{1}(\gamma_n), \qquad  \int_{\CC^n} \vphi d \gamma_n  = \EE \int_{\CC^n} \vphi d \mu_{\infty}, \label{eq_252}
\end{equation} \label{prop_317}
where $L^1(\gamma_n)$ is the collection of all $\gamma_n$-integrable, real valued functions on $\CC^n$.
\end{proposition}

\begin{proof} Since the quadratic polynomial $p_t$ belongs to $\cQ^n$, the measure
 $\mu_t$ belongs to $\cG^n$ for all $t$, according to (\ref{eq_941}). Moreover, by Lemma \ref{lem_220} we know that with probability one,
$$ a_t = a_{\mu_t} \stackrel{t \rightarrow \infty}{\longrightarrow} a_{\infty},
\qquad
 A_{\mu_t} = 2 B_t^{-1} \stackrel{t \rightarrow \infty}{\longrightarrow} A_{\infty}, $$
 where $a_{\infty} \in \CC^n$ and $A_{\infty} \in \overline{M_n^+(\CC)}$ are random variables with $A_{\infty} \leq 2 Id$.
Denote by $\mu_{\infty} \in \cG^n$ the random complex Gaussian with center $a_{\infty}$ and complex covariance matrix $A_{\infty}$.
Then almost-surely, \begin{equation}
 \mu_t \stackrel{t \rightarrow \infty} \longrightarrow \mu_{\infty}. \label{eq_955}
 \end{equation} Since $A_{\infty} \leq 2 Id$, with probability one the complex Gaussian
$\mu_{\infty}$ is more curved than the standard Gaussian measure in $\CC^n$. For any $\vphi \in C(\CC^n)$, it follows from
Proposition \ref{prop_1345}(ii) that
$$ \int_{\CC^n} \vphi d \gamma_n = (2 \pi)^{-n} \EE \int_{\CC^n} \vphi(z) e^{-p_t(z)} dz = \EE \int_{\CC^n} \vphi d \mu_t \stackrel{t \rightarrow \infty}\longrightarrow \EE \int_{\CC^n}
\vphi d \mu_{\infty}, $$
where we used (\ref{eq_955}) as well as the bounded convergence theorem in $(\Omega, \cF, \PP)$ in the last passage.
This proves (\ref{eq_252}) in the case where $\vphi \in C(\CC^n)$. By approximation from below,
we see that (\ref{eq_252}) holds true for all non-negative, $\gamma_n$-integrable functions. This implies
(\ref{eq_252}) is the general case. Next, thanks to our assumption (\ref{eq_239})
and to conclusion (iii) of Proposition \ref{prop_1345}, we may apply Lemma
\ref{lem_955} and obtain that
$$ \lambda_{k+1}(B_t) \geq \frac{\eps}{k+1} \cdot t \qquad \qquad (t \geq 0). $$
Therefore, as $A_{\mu_t} = 2 B_t^{-1}$,
$$ \lambda_{n-k} (A_{\infty}) = \lim_{t \rightarrow \infty} \lambda_{n-k}(A_{\mu_t}) =
\lim_{t \rightarrow \infty} \frac{2}{\lambda_{k+1}(B_t)} = 0. $$
Thus $A_{\infty}$ is a matrix whose rank is at most $k$, and with probability one, the random complex Gaussian $\mu_{\infty}$
is supported in a $k$-dimensional complex, affine
subspace of $\CC^n$.\end{proof}

\section{A martingale supported on the fiber}

We prove the following:

\begin{proposition} Let $1 \leq k \leq n$, let $Z \subseteq \CC^n$ be a closed set and let $U \subseteq \CC^n$ be an open set containing the origin.
Assume that $f: U \rightarrow \CC^k$ is holomorphic, with $f(0) = 0$, such that $Z = \{ z \in U \, ; \, f(z) = 0 \}$.
Assume also that $0$ is a regular value of $f$.

\smallbreak Then there exists a  $\cQ^n$-valued adapted stochastic process $(p_t)_{t \geq 0}$, satisfying (i), (ii) and (iii) from Proposition \ref{prop_1345}, and
moreover the following holds with probability one: For all $t > 0$ we have $$ a_t \in Z \qquad \text{and} \qquad \lambda_{k+1}(M_t) \geq c_{n}, $$
where $c_{n} = 1/n$ and where as before $M_t = B_t^{1/2} \Sigma_t \Sigma_t^* B_t^{1/2}$. \label{prop_1024}
\end{proposition}

The proof requires some preparation. Fix $Z, U \subseteq \CC^n$ and $f = (f_1,\ldots, f_k): \CC^n \rightarrow \CC^k$ as in Proposition \ref{prop_1024}.
The value $0$ is a regular value of $f$, hence for any $z \in Z$, the Jacobian matrix \begin{equation} (\partial f_j / \partial z_\ell)_{1 \leq j \leq k,1 \leq \ell \leq n}
 \label{eq_1033} \end{equation} has rank $k$. By continuity, there exists an open set $\tilde{U} \subseteq U$ containing $Z$
 such that the Jacobian matrix in (\ref{eq_1033}) has rank $k$ throughout $\tilde{U}$. We may replace $U$ by the smaller set
 $\tilde{U}$ and assume from now on that the Jacobian matrix in (\ref{eq_1033}) has rank $k$ throughout $U$.
For a holomorphic function $g: \CC^n \rightarrow \CC$ we write $$ \nabla g = (\partial g / \partial z_1,\ldots, \partial g / \partial z_n)
$$ for the holomorphic gradient, viewed as a row vector.

\begin{lemma} There exists a smooth map $H: U \rightarrow G_{n,k}$
such that
\begin{equation}  \left( \nabla f_{j}(z) \right)^* \in H(z) \qquad \qquad (z \in U, j=1,\ldots,k). \label{eq_1143} \end{equation} \label{lem_1319}
\end{lemma}

\begin{proof} For $z \in U$ write $H(z)$ for the complex subspace spanned by the vectors
$$ (\nabla f_{ 1}(z))^*,\ldots, (\nabla f_{k}(z))^* \in \CC^n. $$
Then $H(z)$ is a $k$-dimensional subspace for any $z \in U$, since the Jacobian matrix in (\ref{eq_1033}) has rank $k$.
Clearly the subspace $H(z) \in G_{n,k}$ varies smoothly with $z \in U$.
\end{proof}

Pick a smooth function $\theta: \CC^n \rightarrow [0,1]$ that is supported on $U$ with
\begin{equation}
 \theta(z) = 1 \qquad \qquad \text{for all} \ z \in Z. \label{eq_954}
\end{equation}
For a quadratic polynomial $p \in \cQ^n$ with $a_p \in U$ we set
$$ \tilde{H}(p) = B(p)^{-1/2} \cdot H(a_p) \in G_{n, k}, $$
where $H: U \rightarrow G_{n,k}$ is as in Lemma \ref{lem_1319}. Thus, we ``rotate'' the subspace $H(a_p)$ using the operator $B(p)^{-1/2}$
in order to obtain the subspace $\tilde{H}(p)$.
For a quadratic polynomial $p \in \cQ^n$ we define
\begin{equation}  \Sigma(p) := \frac{\theta(a_p)}{\sqrt{n}} \cdot B(p)^{-1/2} \cdot \pi_{\tilde{H}(p)}, \label{eq_515} \end{equation}
where we recall that $\pi_{\tilde{H}(p)}$ is the orthogonal projection matrix whose kernel is $\tilde{H}(p)$.
Then $\Sigma(p)$ varies smoothly with $p \in \cQ^n$, and moreover always
\begin{equation} |B_p^{1/2} \cdot \Sigma(p)| \leq 1. \label{eq_1430} \end{equation}
 It follows from Lemma \ref{lem_1319} that if $a_p \in U$ and $j=1,\ldots,k$,
then the vector $(\nabla f_{j}(a_p))^*$ belongs to the kernel of $\pi_{H(a_p)}$. Therefore
$B(p)^{-1/2} (\nabla f_{j}(a_p))^*$ belongs to the kernel of $\pi_{\tilde{H}(p)}$. Consequently, for all $p \in \cQ^n$, if $a_p \in U$ then
\begin{equation}   \nabla f_{j}(a_p) \, \Sigma(p) = 0 \qquad \text{for} \ j=1,\ldots,k. \label{eq_1454} \end{equation}
Thanks to (\ref{eq_1430}) we may apply Proposition \ref{prop_1345}, and conclude the existence of the $\cQ^n$-valued stochastic process $(p_t)_{t \geq 0}$. Recall that we abbreviate $a_t = a_{p_t}, B_t = B_{p_t}$ and $\Sigma_t = \Sigma(p_t)$.

\begin{proof}[Proof of Proposition \ref{prop_1024}]
Consider the following stopping time:
$$ T = \inf  \left \{ t > 0 \, ; \, a_t \not \in U \right \}. $$
Since $a_0 = 0 \in U$ and $U$ is open, the stopping time $T$ is almost-surely positive.
We claim that $a_t \in Z$ for all $t \in [0, T)$. Indeed, recall that $d a_t = \Sigma_t d W_t$. Therefore the quadratic variation $[a_t^j, a_t^\ell]$ vanishes for all $j,\ell=1,\ldots,n$.
This means that when applying the It\^o formula for $d f_{j}(a_t)$, there will be no quadratic variation term as $f_{j}$ is holomorphic.
Thus, for $j=1,\ldots,k$ and $0 < t < T$,
\begin{equation}
  d f_{j}(a_t) = \nabla f_{j}(a_t) \cdot d a_t = \nabla f_{j}(a_t) \cdot \Sigma_t dW_t = \nabla f_{j}(a_{p_t}) \Sigma(p_t) d W_t = 0,
\label{eq_1322} \end{equation}
according to (\ref{eq_1454}). For $t > 0$ abbreviate $t \wedge T = \min \{t, T \}$. It follows from (\ref{eq_1322}) that
\begin{equation}
 f_{j}(a_{t \wedge T}) = f_{j}(a_0) = f_{j}(0) = 0 \qquad \qquad \text{for all} \ t \geq 0, j=1,\ldots,k.
 \label{eq_1319} \end{equation}
 Since $Z = \{ z \in U \, ; \, \forall j\, f_j(z) = 0 \}$, from (\ref{eq_1319}) we deduce that
\begin{equation}
a_t \in Z \qquad \text{for all} \ t \in [0, T).
\label{eq_1359}
\end{equation}
Consider the stopping time
$$ S = \inf  \{ t > 0 \, ; \, a_t \not \in Z \}. $$
From (\ref{eq_1359}) we learn that $S \geq T$ with probability one.
However, since $Z$ is a closed set that is contained in the open set $U$, the continuity of the process $a_t$ implies
that $S < T$ whenever $T$ is finite. Therefore $T = S = +\infty$ almost-surely, and (\ref{eq_1359})
shows that $a_t \in Z$ for all $t > 0$. In particular, for any $t > 0$, by (\ref{eq_954}),
$$ \theta(a_t) = 1. $$
Consequently, for all $t > 0$,
\begin{equation}  B_t^{1/2} \Sigma_t =  \frac{\theta(a_t)}{\sqrt{n}} \pi_{\tilde{H}(p_t)} =  \frac{\pi_{\tilde{H}(p_t)}}{\sqrt{n}} \qquad (t > 0), \label{eq_1021} \end{equation}
where $\tilde{H}(p_t)$ is a certain  complex subspace of dimension $k$.
Since $M_t = B_t^{1/2} \Sigma_t \Sigma_t^* B_t^{1/2}$, we understand from (\ref{eq_1021}) that $M_t \in \overline{M_n^+(\CC)}$ and $\lambda_{k+1}(M_t) = 1/n$ for all $t$.
 This completes the proof of the proposition.\end{proof}

The decomposition of the Gaussian measure $\gamma_n$ alluded to in the Introduction is described in the following theorem:

\begin{theorem} Let $1 \leq k \leq n$, let $Z \subseteq \CC^n$ be a closed set and let $U \subseteq \CC^n$ be an open set containing the origin.
Assume that $f: U \rightarrow \CC^k$ is holomorphic, with $f(0) = 0$, such that $Z = \{ z \in U \, ; \, f(z) = 0 \}$
and such that $0$ is a regular value of $f$.
Then there exists a random complex Gaussian probability measure $\mu \in \cG^n$ with the following properties:
\begin{enumerate}
\item[(i)] With probability one, $\mu$ is more curved than the standard Gaussian measure and it is supported
in a $k$-dimensional complex, affine subspace of $\CC^n$.
\item[(ii)] For any $\vphi \in L^1(\gamma_n)$,
$$ \int_{\CC^n} \vphi d \gamma_n = \EE \int_{\CC^n} \vphi d \mu. $$
\item[(iii)] With probability one, the center $a_{\mu}$ satisfies $a_{\mu} \in Z$.
\end{enumerate} \label{thm_330}
\end{theorem}

\begin{proof} In view of Proposition \ref{prop_1024} we may apply
 Proposition \ref{prop_317} and set $\mu := \mu_{\infty}$.
 Now properties (i) and (ii) follow from Proposition \ref{prop_317}
 while (iii) follows from Proposition \ref{prop_1024}  as
 $$ a_\mu = a_{\infty} = \lim_{t \rightarrow \infty} a_t \in Z $$
 since $a_t \in Z$ for all $t > 0$ and $Z$ is closed.
\end{proof}

\begin{remark}{\rm The point $a_{\mu} = a_{\infty}$ is a random vector supported in $Z$. By setting $\nu(A) = \PP(a_{\mu} \in A)$
for $A \subseteq \CC^n$ we thus obtain a certain probability measure $\nu$ supported on $Z$. Does this probability measure $\nu$ associated with $Z$ have any significance?
In the case where $Z$ is a linear subspace, the measure $\nu$ is a standard Gaussian in $Z$.
}\end{remark}

\begin{remark}{\rm The stochastic process $(a_t)_{t \geq 0}$ has three important properties: The point $a_t$ belongs
to $Z$ at all times, this process is a martingale with $d a_t = \Sigma_t dW_t$, and furthermore
$$ \int_0^{\infty} \Sigma_t^* \Sigma_t dt = Id - \lim_{t \rightarrow \infty} B_t^{-1} = Id - P $$
for some positive-definite matrix $P$ whose rank is the codimension of $Z$. These three properties suffice
for obtaining the conclusion of Theorem \ref{thm_330}.

\smallbreak What is the r\^ole  of the complex structure in our proof?
Can we construct a process with these three properties given a submanifold $Z \subseteq \RR^n$?
This question poses a challenge, even in the case where $Z$ is a minimal surface. A na\"{i}ve approach
towards the real case could be to replace the definition (\ref{eq_515}) of $\Sigma(p)$ by something of the form
$$   \Sigma(p) := \frac{\theta(a_p)}{\sqrt{n}} \cdot B(p)^{-1/2} \cdot \pi_{\tilde{H}(p)} \cdot S, $$
for a certain matrix $S$. The matrix $S$ needs  to satisfy two requirements: First, the It\^o term should vanish in order to ensure that $a_t \in Z$. This
is equivalent to the requirement that
\begin{equation}  \Tr \left[ S S^* \pi_{\tilde{H}(p)} B(p)^{-1/2} \nabla^2 f_j B(p)^{-1/2} \pi_{\tilde{H}(p)} \right] = 0 \qquad (j=1,\ldots,k), \label{eq_910_}
\end{equation}
where $\nabla^2 f_j$ is the Hessian of $f_j: \RR^n \rightarrow \RR$. Note that this requirement automatically holds in the complex case.
Second, some regularity is needed, maybe of the form $c Id \leq S S^* \leq C Id$ for some constants $c, C$ depending on $Z$.
The requirement (\ref{eq_910_}) amounts to linear constraints on the matrix $S S^*$. However, as $B(p_t)^{-1/2}$ is likely to be an almost-degenerate matrix for large $t$,
it is not entirely clear how to satisfy these linear constraints while keeping $S$ and $S^{-1}$ bounded.
}\end{remark}

\section{Complex waist inequalities}

For $v \in \CC^k$ and $R > 0$ write $D(v,R) = \{ z \in \CC^k \, ; \, |z - v| \leq R \}$.
We claim that for any $z_1, z_2 \in \CC^k$,
\begin{equation}
|z_1| \leq |z_2| \qquad \Longrightarrow \qquad \gamma_k(D(z_1,R)) \geq \gamma_k(D(z_2,R)). \label{eq_1130}
\end{equation}
Indeed, consider the function $\rho_R(z) = \gamma_k(D(z,R))$, defined for $z \in \CC^k$. This function is clearly a
radial function in $\CC^k$. By the Pr\'ekopa-Leindler inequality, the function $\rho_R$ is log-concave since it is the convolution of the Gaussian measure with
the characteristic function of a Euclidean ball (see, e.g., \cite[Corollary 1.4.2]{AGM}). Since $\rho_R$ is even and log-concave, it is
necessarily non-increasing on any ray emanating from the origin, proving (\ref{eq_1130}). We shall need the following lemma:

\begin{lemma} Let $\mu$ be a complex Gaussian measure in $\CC^k$ that is more curved
than the standard Gaussian $\gamma_k$. Then for any $v \in \CC^k$ and $R > 0$,
\begin{equation} \int_{D(a_{\mu},R)} e^{\Re(v^* z)} d \mu(z) \geq \gamma_k (D(v, R)) \cdot \int_{\CC^k} e^{\Re(v^* z)} d \mu(z). \label{eq_1157} \end{equation}
 \label{lem_1009}
\end{lemma}

\begin{proof} Translating, we may assume that $a_{\mu} = 0$.
 Let $\nu$ be the probability measure on $\CC^n$ whose density with respect to $\mu$ is proportional to the function $z \mapsto e^{\Re(v^* z)}$. Our goal is to prove that
\begin{equation}\label{eq_1138}
  \nu( D(0, R) ) \geq \gamma_k( D(v, R) ).
\end{equation}
Assume first that the support of $\mu$ spans the entire space $\CC^k$. Then for some $B \in M_n^+(\CC)$, the density of $\mu$ with respect to
the Lebesgue measure is
$$ \frac{d \mu}{d\lambda}(z) = (2 \pi)^{-n} \det(B) \cdot e^{-z^* B z / 2} \qquad \qquad (z \in \CC^n). $$
Since $\mu$ is more curved than the standard Gaussian, necessarily $B \geq Id$. Set $u = B^{-1} v$. Note that
$$ \frac{d \nu}{d\lambda}(z) = (2 \pi)^{-n} \det(B) \cdot e^{-(z - u)^* B (z - u) / 2}. $$
Hence $\nu$ is the push-forward of the standard Gaussian measure $\gamma_k$ under the affine map
$$ T(z) =  u + B^{-1/2} z = B^{-1} v + B^{-1/2} z \qquad \qquad (z \in \CC^n). $$
 Note that $T^{-1}(z) = -B^{-1/2} v + B^{1/2} z$.
Since $B^{1/2} \geq Id$, then $B^{1/2} (D(0,R)) \supseteq D(0, R)$, and hence
$$ T^{-1}(D(0, R)) = -B^{-1/2} v + B^{1/2}(D(0, R)) \supseteq -B^{-1/2} v + D(0, R) = D( -B^{-1/2} v, R). $$
Since $B^{-1/2} \leq Id$ we know that $|B^{-1/2} v| \leq |v|$. We may now use (\ref{eq_1130}) to obtain
$$ \nu( D(0, R) ) = \gamma_k ( T^{-1}( D(0, R) ) ) \geq  \gamma_k( D( -B^{-1/2} v, R) ) \geq \gamma_k(D(v,R)), $$
and (\ref{eq_1138}) is proven. This completes the proof of (\ref{eq_1157}) in the case where the support of $\mu$ spans $\CC^k$.
The general case follows by an approximation argument (e.g., convolve $\mu$ with a small Gaussian).
\end{proof}

\begin{proposition}
Let $1 \leq k \leq n$, let $Z_1 \subseteq \CC^n$ be a non-empty closed set and let $U_1 \subseteq \CC^n$ be an open set.
Suppose  that $f_1: U_1 \rightarrow \CC^k$ is holomorphic and that $Z_1 = \{ z \in U_1 \, ; \, f_1(z) = 0 \}$.
Assume also that $0$ is a regular value of $f_1$. Then for any $R \geq 0$ and $v \in \CC^k$ with $|v| \geq d(0, Z_1)$,
\begin{equation}  \gamma_n( Z_1 + R ) \geq \gamma_k(D(v, R) ). \label{eq_1019}
\end{equation}
\label{prop_1115}\end{proposition}

\begin{proof} The set $Z_1$ is closed and non-empty, hence there exists $z_1 \in Z_1$ with $|z_1| = d(0, Z_1)$.
Set $f(z) = f(z + z_1)$ and $Z = Z_1 - z_1, U = U_1 - z_1$. Thus $Z = \{ z \in U \, ; \, f(z) = 0 \}$,
the set $Z$ contains the origin, and $0$ is a regular value of $f$. We may therefore apply Theorem \ref{thm_330} and conclude that
there exists a certain random complex Gaussian $\mu$ with
\begin{equation} \int_{\CC^n} \vphi d \gamma_n = \EE \int_{\CC^n} \vphi d \mu \qquad \text{for all} \ \vphi \in L^1(\gamma_n)). \label{eq_1325} \end{equation}
In particular, it follows from (\ref{eq_1325}) that
\begin{equation}   \int_{Z + R} e^{-\Re(z_1^* z)} d \gamma_n(z) = \EE \int_{Z + R} e^{-\Re(z_1^* z)} d \mu(z) \geq \EE \int_{D(a_{\mu}, R)} e^{-\Re(z_1^* z)} d \mu(z),
\label{eq_1007} \end{equation}
where we used the fact that $\PP( a_{\mu} \in Z) = 1$ in the last passage, as follows from Theorem \ref{thm_330}(iii).
According to Theorem \ref{thm_330}(i), with probability one the complex Gaussian $\mu$ is supported in a $k$-dimensional,
complex affine subspace, and it is more curved than the standard Gaussian. By  Lemma \ref{lem_1009} and (\ref{eq_1130}),
\begin{equation}  \EE \int_{D(a_{\mu}, R)} e^{-\Re(z_1^* z)} d \mu(z) \geq \gamma_k (D(v, R)) \cdot \EE \int_{\CC^n} e^{-\Re(z_1^* z)} d \mu(z), \label{eq_1015}
\end{equation}
where we used that $|v| \geq |z_1|$. Let us now use (\ref{eq_1325}) with $\vphi(z) = e^{-\Re(z_1^* z)}$. From (\ref{eq_1007}) and (\ref{eq_1015}) we conclude that
\begin{equation}  \int_{Z + R} e^{-\Re(z_1^* z)} d \gamma_n(z) \geq \gamma_k (D(v, R)) \cdot \int_{\CC^n} e^{-\Re(z_1^* z)} d \gamma_n(z). \label{eq_1020}
\end{equation}
We now change variables $w = z + z_1$ in the integrals in (\ref{eq_1020}). We obtain that
$$ \int_{Z_1 + R} e^{-\Re(z_1^* w) - |w - z_1|^2/2} d \lambda(w)  \geq \gamma_k (D(v, R)) \cdot \int_{\CC^n} e^{-\Re(z_1^* w) - |w - z_1|^2/2} d \lambda(w), $$
or equivalently,
$$ \int_{Z_1 + R} e^{-|w|^2/2} d \lambda(w)  \geq \gamma_k (D(v, R)) \cdot \int_{\CC^n} e^{-|w|^2/2} d \lambda(w). $$
The desired inequality (\ref{eq_1019}) is thus proven.
\end{proof}

The conclusion of Proposition \ref{prop_1115}  is stable under {\it upper Hausdorff limits} in $\CC^n$.
We say that a closed set $Z \subseteq \CC^n$ contains the upper Hausdorff limit of a sequence of closed sets $Z_1,Z_2,\ldots \subseteq \CC^n$  if for
any $R, \delta > 0$ there exists $N \geq 1$
with
\begin{equation}
\forall m \geq N, \qquad D(0, R) \cap Z_m \subseteq Z + \delta.
\label{eq_825} \end{equation}

\begin{lemma} Let $1 \leq k \leq n-1$ and let $K \subseteq \CC^n$ be a
convex body.
Assume that the closed set $Z \subseteq \CC^n$
contains the upper Hausdorff limit of a sequence $Z_1,Z_2,\ldots$
of closed sets in $\CC^n$. Then,
 $$ \gamma_n( Z + K ) \geq \limsup_{m \rightarrow \infty} \gamma_n(Z_m + K). $$
\label{cor_621}
\end{lemma}

\begin{proof} Set $L = \limsup_{m \rightarrow \infty} \gamma_n(Z_m + K)$. It suffices to prove that for any fixed  $\eps >0$,
\begin{equation} \gamma_n( Z + K)
\geq L - \eps.
\label{eq_914} \end{equation}
Denote $r = \sup_{z \in K} |z| < \infty$.
In order to prove (\ref{eq_914}) we pick a large number $R > r$ such that \begin{equation}
\gamma_n( \CC^n \setminus D(0, R-r) ) \leq \eps.  \label{eq_921} \end{equation}
Then for any $\delta > 0$, according to  (\ref{eq_825}),
\begin{align*}  \gamma_n( \left[ Z + \delta \right] + K)
& \geq \limsup_{m \rightarrow \infty} \gamma_n\left( \left[ D(0, R)  \cap Z_m \right] + K \right)
\\ & \geq- \eps + \limsup_{m \rightarrow \infty} \gamma_n\left(Z_m + K \right)
= L - \eps. \end{align*}
The set $Z + K$ is closed.
Since $\gamma_n$ is a probability   measure,
$$ \gamma_n( Z + K) = \gamma_n \left( \bigcap_{\delta > 0} [Z + K + \delta] \right) =
\lim_{\delta \rightarrow 0^+} \gamma_n \left( \left[ Z + \delta \right] + K \right)
\geq L - \eps, $$
completing the proof of (\ref{eq_914}).
\end{proof}

Assume that
the sequence of continuous functions $f_1,f_2,\ldots: \CC^n \rightarrow \CC^k$ converges to a limit function $f: \CC^n \rightarrow \CC^k$ uniformly on  compacts in $\CC^n$. Then $Z = f^{-1}(0)$ contains
the upper Hausdorff  limit
of the sequence $(Z_m)_{m \geq 1}$, where  $Z_m = f_m^{-1}(0)$. For a short argument see, e.g.,
the proof of Theorem 4.2 in \cite{K_waist}.

\begin{proof}[Proof of Theorem \ref{thm_1048}] There exists a point $p = (p_1,\ldots,p_n) \in Z$ with $|p| = d(0, Z)$.
For $\delta > 0$ denote
$$ g_{\delta}(z) = f(z) + \delta (z_1 - p_1,\ldots,z_k - p_k) \in \CC^k, \qquad (z \in \CC^n). $$
For all but finitely many values of $\delta$, the point $p$ is a regular point of the smooth map $g_{\delta}: \CC^n \rightarrow \CC^k$.
Pick a sequence $\delta_m \searrow 0$ such that $p$ is a always a regular point of $g_{\delta_m}$.
Then for each $m$, the image $g_{\delta_m}( D(p, 1/m) )$ contains an open neighborhood of the origin in $\CC^k$.
By Sard's theorem,
for any $m$ we may find $z_m \in D(p, 1/m) \subseteq \CC^n$ such that $g_{\delta_m}(z_m)$ is a regular value
of $g_{\delta_m}$ with $|g_{\delta_m}(z_m)| \leq 1/m$. Define,
$$
h_m(z) = g_{\delta_m}(z + z_m - p) - g_{\delta_m}(z_m) \qquad \qquad (z \in \CC^n, m \geq 1). $$
Then $0 = h_m(p)$ is a regular value of the holomorphic function $h_m: \CC^n \rightarrow \CC^k$. Moreover, $h_m \longrightarrow f$ as $m \rightarrow \infty$ uniformly on compacts in $\CC^n$. By the remark preceding the proof,
the set $Z = f^{-1}(0)$ contains the upper Hausdorff limit of $(Z_m)_{m \geq 1}$, where $Z_m = h_m^{-1}(0)$. Moreover, $p \in Z_m$ and hence for any $m \geq 1$,
\begin{equation}  d(0, Z_m) \leq |p|. \label{eq_623} \end{equation}
Let us now apply
 Proposition \ref{prop_1115} for the non-empty closed set $Z_m$, the open set $U = \CC^n$ and the holomorphic function $h_m: U \rightarrow \CC^k$ for which $0$ is a regular value.
 By (\ref{eq_623}) and the conclusion of the proposition, for any $m \geq 1$ and $R > 0$,
\begin{equation}
\gamma_n( Z_m + R ) \geq \gamma_k(D(v, R))) = \gamma_n(E + R),
\label{eq_601}
\end{equation}
where $v \in \CC^k$ is any point with $|v| = |p|$. Now
(\ref{eq_1608}) follows from (\ref{eq_601}) and Lemma \ref{cor_621}, completing the proof.
\end{proof}

\begin{proof}[Proof of Corollary \ref{thm_1237}]  Set $r_K = \inf \{ |z| \, ; \, z \in \partial K \}$.
The infimum is attained by compactness, hence
there exists $z_0 \in \partial K$ such that $ |z_0| =  r_K$. Clearly,
\begin{equation}
r_K B^n \subseteq K. \label{eq_437}
\end{equation}
Since the origin belongs to the interior of $K$, necessarily $z_0 \neq 0$. Setting $H = z_0^{\perp}
= \{ z \in \CC^n \, ; \, z_0^* z = 0 \}$, we claim that
\begin{equation}
K \subseteq H + r_K B^n. \label{eq_430}
\end{equation}
Indeed, the Euclidean sphere of radius $r_K$ centered at the origin is tangent to $\partial K$ at the point $z_0$. By convexity, $K$ lies to one side of the (real) tangent hyperplane to that sphere at the point $z_0$. Equivalently,
$$ \sup_{z \in K} \Re (z_0^* z) = z_0^* z_0 = |z_0|^2 = r_K^2. $$
Since $K$ is circled, then $|z_0^* z| \leq r_K^2$ for all $z \in K$ and (\ref{eq_430}) follows. In particular, $K + H \subseteq H + r_K B^n$.
Let $f: \CC^n \rightarrow \CC$ be a holomorphic function, set $Z = f^{-1}(0)$ and assume that $Z \neq \emptyset$. Let $t \in \RR$ be defined via
$$ d(0, t z_0 + H) = d(0, Z). $$
We now apply the inclusion (\ref{eq_437}) and Theorem \ref{thm_1048} in order to obtain that for any $r > 0$,
$$  \gamma_n(Z + r K) \geq \gamma_n(Z + r r_K B^n) \geq \gamma_n(t z_0 + H + r r_K B^n) \geq \gamma_n(t z_0 + H + rK), $$
where we also used (\ref{eq_430}). The desired conclusion follows with $H_1 = t z_0 + H$.
\end{proof}

We end this paper with a real analog of Corollary \ref{thm_1237}. With a slight abuse of notation, in the next corollary we write $\gamma_n$
for the standard Gaussian measure on $\RR^n$ whose density is $x \mapsto (2\pi)^{-n/2} \exp(-|x|^2/2)$.

\begin{corollary} Let $K \subseteq \RR^n$ be a convex body with $K = -K$. Then there exists an $(n-1)$-dimensional
subspace $H \subseteq \RR^n$ with the following property: For any continuous function $f: \RR^n \rightarrow \RR$
there exists $t \in \RR$ such that $L = f^{-1}(t)$ satisfies
$$ \gamma_n(L + r K) \geq \gamma_n(H + r K) \qquad \text{for all} \ r > 0. $$
\end{corollary}

\begin{proof} Set $r_K = \inf \{ |x| \, ; \, x \in \partial K \}$, and let $x_0 \in \partial K$ be such that $|x_0| = r_K$.
Then $K \supseteq r_K B^n$ while for $H = x_0^{\perp} \subseteq \RR^n$ we have $K \subseteq H + r_K B^n$.
According to Gromov's Gaussian waist inequality (\cite{gromov}, see also \cite{K_waist}),
for any continuous function $f: \RR^n \rightarrow \RR$
there exists $t \in \RR$ such that $L = f^{-1}(t)$ satisfies
\begin{equation*}
\gamma_n(L + r K)  \geq \gamma_n(L + r \cdot r_K B^n)  \geq \gamma_n(H + r \cdot r_K B^n) \geq \gamma_n(H + r K). \tag*{\qedhere}
\end{equation*}
\end{proof}

\begin{remark}{\rm The requirement that $Z = f^{-1}(0)$ in Theorem \ref{thm_1048} may probably be replaced by the requirement that $Z \subseteq \CC^n$ is an irreducible complex-analytic variety of dimension $n-k$. One way to prove this is to solve the stochastic differential equations
in (\ref{eq_826}) under the apriori assumption that $a_t$ is confined to $Z$. That is, if $Z$ is smooth then  we may define $a_t = \Phi(b_t)$ where $\Phi$ is an embedding map
and $b_t$ is a stochastic process in an abstract complex manifold $X$. In the general case, the complex variety $X$ is smooth outside a singular set of smaller dimension, which with probability
one is avoided by the process $b_t$.
}\end{remark}

{
}

\smallbreak
\noindent Department of Mathematics, Weizmann Institute of Science, Rehovot 76100 Israel, and
School of Mathematical Sciences, Tel Aviv University, Tel Aviv 69978 Israel.

\smallbreak
\hfill \verb"boaz.klartag@weizmann.ac.il"

\end{document}